\documentclass[10pt,reqno]{amsart}
\usepackage[english]{babel}
\usepackage{amsrefs}
\usepackage[misc]{ifsym}
\usepackage{mathrsfs,mathtools}
\usepackage{esint} 
\usepackage{faktor} 
\usepackage[utf8]{inputenc} 
\usepackage[left=2.9cm,right=2.9cm,top=2.8cm,bottom=2.8cm]{geometry}
\usepackage{graphicx}
\usepackage{enumitem}
\usepackage{todonotes, cancel}
\usepackage[dvipsnames]{xcolor}
\usepackage[colorlinks=true,urlcolor=blue,citecolor=blue,linkcolor=blue,linktocpage,pdfpagelabels,bookmarksnumbered,bookmarksopen]{hyperref}

\usepackage{type1cm}

\newtheorem{theorem}{Theorem}[section]
\newtheorem{lemma}{Lemma}[section]
\newtheorem{prop}[lemma]{Proposition}

\newtheorem{defi}[lemma]{Definition}
\theoremstyle{remark}
\newtheorem{rmk}[lemma]{Remark}
\newtheorem{ex}[lemma]{Example}
\numberwithin{equation}{section}

\newcommand{\R}{\mathbb R}
\newcommand{\N}{\mathbb N}

\allowdisplaybreaks

\begin{document}
\title[Regularizing Effects for $X$-elliptic equations]{Regularizing effect of the interplay between coefficients in linear and semilinear $X$-elliptic equations}

\author[P. Malanchini]{Paolo Malanchini}
\address[Paolo Malanchini]{Dipartimento di Matematica e Applicazioni\\ Universit\`a degli Studi di Milano - Bicocca, via Roberto Cozzi 55, 20125 - Milano, Italy}
\email{p.malanchini@campus.unimib.it}

\author[G. Molica Bisci]{Giovanni Molica Bisci}
\address[Giovanni Molica Bisci]{Department of Human Sciences and Promotion of Quality of Life, San Raffaele University, via di Val Cannuta 247, I-00166 Roma, Italy} 
\email{giovanni.molicabisci@uniroma5.it}

\author[S. Secchi]{Simone Secchi}
\address[Simone Secchi]{Dipartimento di Matematica e Applicazioni\\ Universit\`a degli Studi di Milano - Bicocca, via Roberto Cozzi 55, 20125 - Milano, Italy}
\email{simone.secchi@unimib.it}


\begin{abstract}
We study the regularizing effect arising from the interaction between the coefficient \(a\) of the zero order term and the datum \(f\) in the problem
\begin{equation*}
\left\{ 
    \begin{alignedat}{2}
        -\mathcal{L}u + a(x) g(u) & = f(x) \quad &&\mbox{in} \;\; \Omega, \\ 
        u &= 0 \quad &&\mbox{on} \;\; \partial\Omega, 
    \end{alignedat} 
\right.     
\end{equation*}
where $\Omega\subseteq\R^N$ is a bounded domain and $\mathcal{L}$ is an $X$-elliptic operator introduced by Lanconelli and Kogoj in \cite{KL00}.
If $f \in L^1(\Omega)$, we prove that the \(Q\)-condition
introduced by Arcoya and Boccardo in \cite{AB15} is sufficient to ensure the existence and boundedness of solutions in the framework of $X$-elliptic operators as well. Finally, we prove the existence of a bounded solution for linear problems under a more general condition between $f$ and $a$.
\end{abstract}

\maketitle
 

\section{Introduction}

Let $\Omega$ be a bounded domain of $\mathbb{R}^N$, $N \geq 2$, and let $\mathcal{L}$ be the second order partial differential operator defined, for a smooth function $u\colon\R^N\to\R$, by
\begin{equation*}
    \mathcal{L} u\coloneqq \sum_{i,j = 1}^N \partial_i \left(b_{ij}(x) \partial_j u\right), \quad \partial_j = \frac{\partial}{\partial x_j},
\end{equation*}
where $b_{ij} = b_{ij}$ are measurable functions defined on $\R^N$.
We assume that the operator $\mathcal{L}$ is $X$-elliptic with respect to the family of vector fields $X = (X_1,\dots, X_m)$ in the sense of Definition \ref{def_X_elliptic}.

In the first part of this article, we consider the semilinear boundary value problem
\begin{equation} \tag{$P$}\label{problem}
\left\{ 
    \begin{alignedat}{2}
        -\mathcal{L}u + a(x) g(u) & = f(x) \quad &&\mbox{in} \;\; \Omega, \\ 
        u &= 0 \quad &&\mbox{on} \;\; \partial\Omega.
    \end{alignedat} 
\right. 
\end{equation}
With respect to the coefficient $a$ and to the datum $f$, we impose that 
\begin{equation}\tag{H}
\label{eq_cond_a}
a \in L^1(\Omega), \quad a \geq 0, \quad f \in L^1(\Omega).
\end{equation}
As a lower order term, we consider a function $g\colon\R\to\R$ such that
\begin{equation}\tag{H$_g$}
    \label{eq_cond_g}
    g~ \hbox{is continuous, odd and increasing}.
\end{equation}
\begin{rmk} \label{rmk:1.1}
    By \eqref{eq_cond_g}, the function $g$ has an inverse $g^{-1}$ defined in the open interval 
\begin{displaymath}
\left(-\lim_{s\to+\infty}g(s), \lim_{s\to+\infty}g(s)\right).
\end{displaymath}
\end{rmk}
The $Q$-condition of Arcoya and Boccardo was introduced in \cite{AB15} in the form
\begin{equation}
\label{eq_Q_cond}\tag{Q}
\text{there exists } Q \in \left(0,\lim_{s\to+\infty} g(s)\right) \text{ such that } |f(x)| \leq Q a(x) \text{ for a.e. } x \in \Omega
\end{equation}
in order to deal with the problem
\begin{equation} \label{problem_elliptic}
\left\{
\begin{alignedat}{2}
-\operatorname{div}\left( M(x)\nabla u \right)+ a(x) g(u) & = f(x) \quad &&\text{in}   \;\; \Omega, \\
u &= 0 \quad &&\text{on} \;\; \partial\Omega,
\end{alignedat}
\right.
\end{equation}
where $\Omega\subseteq\R^N$ is a bounded open set and $M$ is a bounded elliptic matrix.
As remarked in \cite{AB21}, the rather weak assumption $f \in L^1(\Omega)$
is neither sufficient to ensure the existence of a weak solution of \eqref{problem_elliptic}, nor can we ensure the existence of a solution with finite energy.

By introducing the ``$Q$-condition" \eqref{eq_Q_cond} on the right-hand side $f$, Arcoya and Boccardo were able to construct a unique weak solution $u \in H^1_0(\Omega) \cap L^\infty(\Omega)$ to \eqref{problem_elliptic}.

A few years later, the same authors extended in \cite{AB21} their previous result, giving an explicit $L^\infty$-bound of the solution of \eqref{problem_elliptic} and proving that a maximum principle holds when the inequality \eqref{eq_Q_cond} becomes an equality.

Motivated by \cite{AB15}, many authors have explored the regularizing effect of the $Q$-condition. See, for example, \cite{AB21b} where the results are applied to some Hamilton-Jacobi equations and \cite{ABO20} for data of the type $-\operatorname{div}(F(x))$ with vector-valued function $F(x)\in (L^2(\Omega))^N$.
We also refer to \cites{CAAM23, A23, ZFG24, ZZ17} for more recent contributions on this topic.

Finally, in \cite{ADLSV25} a nonlocal version of problem \eqref{problem_elliptic} is studied. Results analogous to those in \cite{AB15} are obtained, but in the setting of a nonlocal operator, which generalizes the fractional Laplace operator. Our aim is to construct a bounded weak solution of \eqref{problem} under the assumption \eqref{eq_Q_cond} in the framework of $X$-elliptic operators. 

\medskip
In the second part of the article, we consider the linear boundary problem
\begin{equation}\label{problem_linear}\tag{$\tilde{P}$}
\left\{ 
    \begin{alignedat}{2}
        -\mathcal{L}u + a(x) u & = f(x) \quad &&\mbox{in} \;\; \Omega, \\ 
        u &= 0 \quad &&\mbox{on} \;\; \partial\Omega, 
    \end{alignedat} 
\right. 
\end{equation}
with the same hypotheses \eqref{eq_cond_a}-\eqref{eq_cond_g}, but we replace condition \eqref{eq_Q_cond} with the weaker condition:
\begin{equation}
\label{eq_Q_cond_2} \tag{$\tilde{Q}$}
\text{there exist $Q\in  (0,+\infty)$ and $R \in L^r(\Omega)$ such that $|f(x)| \leq Q a(x) + R(x)$ for a.e.} x \in \Omega,
\end{equation}
for suitable values of $r\in (1,+\infty]$, see Section \ref{sec_linear}.

Clearly, assumption \eqref{eq_Q_cond_2} reduces to \eqref{eq_Q_cond} when $R \equiv 0$. This ``generalized $Q$-condition" was introduced in \cite{AB18}, where the authors proved the existence of a weak solution for 
\begin{equation*}
\left\{
\begin{alignedat}{2}
-\operatorname{div}\left( M(x)\nabla u \right)+ a(x) u & = f(x) \quad &&\text{in}   \;\; \Omega, \\
u &= 0 \quad &&\text{on} \;\; \partial\Omega,
\end{alignedat}
\right.    
\end{equation*}
under the condition \eqref{eq_Q_cond_2}, where $M$ is again a bounded elliptic matrix.

\medskip
We extend the results of \cites{AB15, AB18} to the broader framework of $X$-elliptic operators, introduced by Lanconelli and Kogoj in \cite{KL00}. Although these operators are not uniformly elliptic, the notion of $X$-ellipticity is sufficient to ensure the existence of a bounded solution of our problems. To the best of our knowledge, this is the first attempt to investigate the regularizing effect of the $Q$-condition within the setting of $X$-elliptic operators; a framework that includes a wide class of differential operators, see \cite[Section 2]{KS14} for some relevant examples.

\medskip
The paper is organized as follows. In Section~\ref{sec_defn} we introduce the $X$-elliptic operators along with the associated functional spaces. In Section~\ref{sec_semilinear}, we establish the existence and uniqueness of a bounded weak solution of \eqref{problem}, providing additionally an explicit bound for its $L^\infty$-norm.
In Section \ref{sec_linear}, we consider the linear problem \eqref{problem_linear} with the condition \eqref{eq_Q_cond_2}. Under suitable assumptions on the term $R(x)$ in \eqref{eq_Q_cond_2}, we can guarantee the existence and the boundedness of the (unique) solution.

\subsection{Notation}
For each~$q\in[1,+\infty]$, $L^q(\Omega)$ stands for the standard Lebesgue space, whose norm will be indicated with $\|\cdot\|_{L^q(\Omega)}$. For each $N$-dimensional Lebesgue measurable set $A\subseteq\R^N$, the symbol $|A|$ denotes the Lebesgue measure of $A$, and we write $\chi_A$ to indicate its characteristic function. 

We abbreviate $\{u>v\} = \{x\in\R^N: \, u(x)>v(x)\}$, and similarly for $\{u<v\}$, etc. To simplify notation we will often write $\int_{\{u>v\}}$ instead of $\int_{u>v}$ etc.  

For each~$k>0$ we will use the affine function
\begin{equation*}
G_k(s) \coloneqq 
\begin{cases}
    0, & \text{if } |s| \leq k, \\
    s - k, & \text{if } s > k, \\
    s + k, & \text{if } s < -k,
\end{cases}      
\end{equation*}
and the truncation function $T_k(s) \coloneqq s- G_k(s)$, that is
\begin{equation*}
T_k(s) =
\begin{cases}
    s \quad & \hbox{if $|s|< k$},\\
     \operatorname{sgn}(s) k \quad & \hbox{if $|s|\ge k$}.
\end{cases}
\end{equation*}

\section{$X$-elliptic operators and functional setting}
\label{sec_defn}
Lanconelli and Kogoj introduced in \cite{KL00} a new family of degenerate elliptic operators, whose degeneracy is controlled by a family $X$ of vector fields with suitable properties. 

More precisely, we consider the operator $\mathcal{L}$ defined on a smooth function $u\colon\R^N\to\R$ by
\begin{equation}
    \label{eq_def_operator}
    \mathcal{L} u\coloneqq \sum_{i,j = 1}^N \partial_i \left(b_{ij}(x) \partial_j u\right), \quad \partial_j = \frac{\partial}{\partial x_j},
\end{equation}
where $b_{ij}(x) = b_{ij}(x)$ are measurable functions in $\R^N$.

We consider a family $X = \{X_1,\dots, X_m\}$ of vector fields in $\R^N$, $X_j = (c_{j1}, \dots, c_{jN})$, $j=1,\dots m$, where the $c_{jk}$'s are locally Lipschitz continuous functions on $\R^N$. We identify the vector field $X_j$ with the first order differential operator
\begin{equation*}
X_j = \sum_{k=1}^N c_{jk} \partial_k.    
\end{equation*}
\begin{defi}
\label{def_X_elliptic}
We say that the operator $\mathcal{L}$ defined in \eqref{eq_def_operator} is \emph{uniformly $X$-elliptic} in an open subset $\Omega\subseteq\R^N$ if there exists a constant $\Lambda>0$ such that
\begin{equation}
    \label{eq_X_ell}
       \frac{1}{\Lambda} \sum_{j=1}^m \langle X_j(x),\xi\rangle^2 \le \sum_{i,j=1}^N b_{ij}(x)\xi_i\xi_j \le \Lambda \sum_{j=1}^m \langle X_j(x),\xi\rangle^2 
\end{equation}
for all $\xi=(\xi_1,\dots,\xi_N)\in\R^N$ and $x\in\Omega$, where
\begin{equation*}
\langle X_j(x),\xi\rangle \coloneqq \sum_{k=1}^N c_{jk}(x)\xi_k, \quad j=1,\dots, m.
\end{equation*} 
\end{defi}
We denote by $X u$ the \emph{$X$-gradient} of $u$, that is, $X u = (X_1 u, \dots, X_m u)$. Let $\Omega$ be a bounded open subset of $\mathbb{R}^N$. The functional $u \mapsto \|X u\|_{L^2(\Omega)}$ defines a norm on $C^1_0(\Omega)$, and we define the Hilbert space $\mathscr{H}_X(\Omega)$ as the completion of $C^1_0(\Omega)$ with respect to this norm. We set
\begin{equation*}
   \|u\|_X \coloneqq \|Xu\|_{L^2(\Omega)}, 
\end{equation*}
and the scalar product in $\mathscr{H}_X(\Omega)$ given by\footnote{Here and in the rest of the paper we will omit the dot in scalar products for finite-dimensional vector spaces. Hence $Xu Xv$ will stand for $Xu \cdot Xv$, and so on.}
\begin{equation*}
  \langle u, v\rangle_X\coloneqq \int_\Omega Xu Xv.  
\end{equation*}
In this article we assume, as in \cite[p. 409]{KS14}, a Sobolev-type embedding result:
\begin{enumerate}[label={\textbf{(S)}}]
    \item \label{sobolev}  There exists a number $2^*_X = 2^*_X(\Omega) > 2$ such that the embedding
\[
\mathscr{H}_X(\Omega) \hookrightarrow L^p(\Omega)
\]
is continuous for $p \in [1, 2^*_X]$, and compact for every $p \in [1, 2^*_X)$.
\end{enumerate}
We set
\begin{equation}
    \label{eq_dimension}
    N_X \coloneqq \frac{2 \cdot2^*_X}{2^*_X-2},
\end{equation}
so that $2^*_X = \frac{2N_X}{N_X-2}.$
\begin{ex}
    When the operator $\mathcal{L}$ is the classical Laplace operator $\Delta$, $2^*_X$ corresponds to the usual critical Sobolev exponent, $2^* = \frac{2N}{N-2}$ and $N_X = N$, the dimension of the euclidean space $\R^N$. 
\end{ex}
\begin{ex}
    A nontrivial example of $X$-elliptic operator is provided by the Baouendi-Grushin operator defined for $\gamma\ge 0$ as
    \begin{equation*}
          \Delta_x + |x|^{2\gamma}\Delta_y , \quad (x,y)\in\R^m\times\R^\ell.
    \end{equation*}
    Here $2^*_X = \frac{2N_\gamma}{N_\gamma-2}$, where $N_\gamma = m +(1+\gamma)\ell$, the \emph{homogeneous dimension} associated to the decomposition $\R^N = \R^m\times\R^\ell$.
\end{ex}
Assumption \ref{sobolev} allows us to define the best constant in the Sobolev embedding by
\begin{equation}
	\label{eq_sob_const}
	S \coloneqq \inf_{\substack{ u \in \mathscr{H}_X(\Omega) \\ u \neq 0}} \frac{\| u\|_X^2 }{ \| u \|_{L^{2^*_X}(\Omega)}^2}.
\end{equation}
The notion of $X$-elliptic operators was introduced in \cite{KL00}, where the authors established a Harnack inequality for solutions of $\mathcal{L}u = 0$. Gutiérrez and Lanconelli in \cite{GL03} proved a maximum principle for $X$-elliptic operators with lower-order terms. In the special case where the underlying vector fields $X$ are dilation invariant, they also obtained nonhomogeneous Harnack inequalities and Liouville-type theorems. A further refinement was provided in \cite{KL09}, where a one-sided Liouville-type property was proved.

The well-posedness and long-time behavior of solutions of equations involving $X$-elliptic operators were analyzed in \cite{KS14}. In the context of boundary regularity, the Wiener criterion for the Dirichlet problem was established in \cite{TU15}, while a nonhomogeneous Harnack inequality was further explored in \cite{U15}. More recently, the dynamics of stochastic parabolic equations governed by $X$-elliptic operators were studied in \cite{CLS19}. In \cite{GM23}, Wong–Zakai approximations for non-autonomous stochastic parabolic equations involving $X$-elliptic operators were investigated.
Finally, in the very recent \cites{P25,P25b} existence, uniqueness and regularity for nonlinear problems with $X$-elliptic operators have been treated.

\section{Semilinear problems}
\label{sec_semilinear}
To properly define a weak solution of \eqref{problem} observe that, since that matrix $B(x) \coloneqq (b_{ij}(x))$ is symmetric and positive semi-definite for all $x\in\Omega$, we have by \eqref{eq_X_ell}
\begin{align*}
   \left\lvert \int_\Omega B(x)\nabla u\nabla v \right\rvert &\le \int_\Omega | B(x)\nabla u \nabla u|^{1/2} \cdot| B(x)\nabla v \nabla v|^{1/2}\\& \le \Lambda\int_\Omega |X u| |X v| \le \Lambda \|Xu\|_{L^2(\Omega)} \|Xv\|_{L^2(\Omega)}
\end{align*}
for all $u, v\in C^1_0(\Omega)$. The bilinear form $(u,v)\mapsto \int_\Omega  B(x)\nabla u\nabla v$ can therefore be extended continuously to $\mathscr{H}_X(\Omega)\times \mathscr{H}_X(\Omega)$.

Moreover, condition \eqref{eq_X_ell} also yields the following ellipticity estimate
\begin{equation}
    \label{eq_elliptic}
    B(x)\nabla u \nabla u \ge \frac{1}{\Lambda} |Xu|^2 \quad \mbox{for all $u\in\mathscr{H}_X(\Omega)$ and $x\in\Omega$}.
\end{equation}
The next technical lemma will be useful.
\begin{lemma}
    \label{lemma_convergence}
    Let $(u_j)\subseteq \mathscr{H}_X(\Omega)$ be a sequence such that $u_j\rightharpoonup u$ in $\mathscr{H}_X(\Omega)$ as $j\to+\infty$ for some $u\in\mathscr{H}_X(\Omega)$. Then
    \begin{equation*}
         \int_\Omega B(x)\nabla u_j\nabla v\to \int_\Omega B(x)\nabla u\nabla v\quad\hbox{as $j\to+\infty$,}   
    \end{equation*}
    for all $v\in\mathscr{H}_X(\Omega)$.
\end{lemma}
\begin{proof}
    The proof is straightforward. Using the previous estimate \eqref{eq_elliptic} we get
    \begin{align*}
            0 &= \lim_{j\to+\infty} \int_\Omega |X(u_j-u)| |Xv|\ge \frac{1}{\Lambda} \left\lvert \int_\Omega B(x) \nabla (u_j-u)\nabla v\right\rvert \\&= \frac{1}{\Lambda} \left\lvert \int_\Omega B(x) \nabla u_j\nabla v - \int_\Omega B(x) \nabla u\nabla v \right\rvert ,
    \end{align*}
    for all $v\in\mathscr{H}_X(\Omega)$.
\end{proof}
\begin{defi}
\label{def_weak_sol}
   We say that a function $u\in\mathscr{H}_X(\Omega)$ is a \emph{weak solution} of \eqref{problem} if $a\cdot g(u)\in L^1(\Omega)$ and
    \begin{equation*}
    \int_\Omega  B(x)\nabla u\nabla v + \int_\Omega a(x) g(u) v= \int_\Omega f(x) v
    \end{equation*}
for all $v\in \mathscr{H}_X(\Omega)\cap L^\infty(\Omega).$ 
\end{defi}

\begin{theorem}
\label{thm_existence_1}
Assume that \eqref{eq_cond_a}-\eqref{eq_cond_g} and \eqref{eq_Q_cond} are satisfied. There exists a unique weak solution $u\in\mathscr{H}_X(\Omega)\cap L^\infty(\Omega)$ of \eqref{problem}. Moreover, the following \textit{a priori} estimates holds:
    \begin{equation}
    \label{eq_bound_u}
        \|u\|_{L^\infty(\Omega)} \le g^{-1}(Q).
    \end{equation}
\end{theorem}
The proof of Theorem \ref{thm_existence_1} is based on an approximation scheme. We introduce the sequence of approximated problems 
\begin{equation} \tag{$P_n$}\label{problem_approx}
\left\{
\begin{alignedat}{2}
-\mathcal{L}u + a_n(x) g(u)  & = f_n(x) \quad &&\mbox{in} \;\; \Omega, \\
u &= 0 \quad &&\mbox{on} \;\; \partial\Omega, 
\end{alignedat}
\right.
\end{equation}
where the functions $f_n$ and $a_n$ are defined by
\begin{equation}\label{eq_func_approx}
f_n(x)\coloneqq \frac{f(x)}{1+\frac{1}{n}|f(x)|},   \quad a_n(x) \coloneqq \frac{a(x)}{1+\frac{Q}{n} a(x)},    
\end{equation}
for all $n\in\N$. Clearly $f_n \to f$ and $a_n \to a$ pointwise. 
The $Q$-condition \eqref{eq_Q_cond} holds also for $f_n$ and $a_n$: since the function $s\mapsto s/\left(1+\frac{s}{n}\right)$ is increasing, by \eqref{eq_Q_cond} it follows that
\begin{equation}
    \label{eq_Q_cond_approx}
    |f_n(x)| = \frac{|f(x)|}{1+\frac{1}{n} |f(x)|} \le \frac{Q a(x)}{1+\frac{Q}{n} a_n(x)} = Q a_n(x).
\end{equation}
As a first step, we show that \eqref{problem_approx} is solvable for each $n \in \mathbb{N}$.
\begin{prop}
\label{prop_approx}
    For all $n\in\N$ there exists a weak solution $u_n$ of \eqref{problem_approx}, namely $u_n\in\mathscr{H}_X(\Omega)$ such that
    \begin{equation*}
        \int_\Omega B(x) \nabla u_n \nabla v + \int_\Omega a_n(x) g(u_n) v = \int_\Omega f_n(x) v, \quad \forall v\in \mathscr{H}_X(\Omega).    
    \end{equation*}
\end{prop}
\begin{proof}
We set
\begin{equation*}
 G(t)\coloneqq\int_0^t g(\tau) \,\mathrm{d}\tau.   
\end{equation*}
Let us consider the functional~$J_n$ defined, for any $u\in\mathscr{H}_X(\Omega)$, as
\begin{equation}\label{eq_functional}
J_n(u)\coloneqq \begin{cases}
\dfrac{1}{2}\displaystyle\int_{\Omega} B(x)\nabla u \nabla u + \int_\Omega a_n(x) G(u) - \int_\Omega f_n(x) u \quad &\mbox{ if } a_n G(u)  \in L^1(\Omega),\\
+\infty \quad &\mbox{ if } a_n G(u) \notin L^1(\Omega),
\end{cases}
\end{equation}
so that the solutions of \eqref{problem_approx} coincide with critical points 
of $J_n$. We now split the proof into several steps, investigating the properties of the functional $J_n$.

\medskip
\noindent\textbf{Step 1.} If $a_n G(u) \in L^1(\Omega)$ then $|J_n(u)|<+\infty$.
Simply estimating each term in the functional $J_n$ we get
\begin{equation*}
\begin{aligned}
&\int_\Omega B(x)\nabla u \nabla u \underset{\eqref{eq_elliptic}}{\le} \Lambda \|u\|_X^2 < +\infty, \quad
\left\lvert \int_\Omega a_n(x) G(u) \right\rvert \le \int_\Omega \left\lvert a_n(x) G(u)\right\rvert < +\infty,
\end{aligned}    
\end{equation*}
and
\begin{equation*}
\left\lvert \int_\Omega f_n(x) u \right\rvert \le \|f_n\|_{L^2(\Omega)}\|u\|_{L^2(\Omega)} < +\infty,    
\end{equation*}
by H\"{o}lder's inequality and the Sobolev embedding \ref{sobolev}.

\noindent\textbf{Step 2.} $J_n$ is weakly lower semicontinuous, in the sense that for every sequence $(u_j)\subseteq \mathscr{H}_X(\Omega)$ such that $u_j\rightharpoonup u$ for some $u\in \mathscr{H}_X(\Omega)$ as $j\to+\infty$ we have
\begin{equation*}
\liminf_{j\to+\infty} J_n(u_j)\ge J_n(u).    
\end{equation*}
Suppose by contradiction that there exists a sequence $(u_j)$ such that $u_j\rightharpoonup u$ for some $u\in \mathscr{H}_X(\Omega)$ for $j\to+\infty$ and such that
\begin{equation}
    \label{eq_limit}
    \lim_{j\to+\infty} J_n(u_j) < J_n(u).
\end{equation}
In particular,
\begin{equation*}
\lim_{j\to+\infty } J_n(u_j) < +\infty,     
\end{equation*}
and so by the previous step, $a_n G(u_j)\in L^1(\Omega)$ for $k$ sufficiently large. By the compact embeddings \ref{sobolev}, there exists a subsequence --- still denoted by $(u_j)$ ---  which converges strongly to $u$ in $L^2(\Omega)$ and almost everywhere in $\Omega$.

Since $g$ is odd and increasing, it follows that 
\begin{equation}
    \label{eq_positive}
    G(t)\ge 0 \quad\hbox{for any $t\in\R$}. 
\end{equation}
So, since $a_n(x)\ge 0$ for all $x\in\Omega$ and $G(t)$ is continuous, from Fatou's Lemma we have that
\begin{equation}
\label{eq_convergence1}
 \liminf_{j\to+\infty} \int_\Omega a_n(x) G(u_j) \ge \int_\Omega a_n(x) G(u).   
\end{equation}
Moreover, since $f_n\in L^2(\Omega)$, the strong convergence of $u_j$ to $u$ in $L^2(\Omega)$ implies that
\begin{equation}
\label{eq_convergence2}
\lim_{j\to+\infty} \int_\Omega f_n(x) u_j = \int_\Omega f_n(x) u.   
\end{equation}
Joining together the previous estimates, recalling the definition of $J_n$ in \eqref{eq_functional} and by \eqref{eq_elliptic} we get
\begin{align}
    \label{eq_functional_limit}
    \lim_{j\to+\infty} J_n(u_j) &\ge \liminf_{j\to+\infty} \left( \frac{1}{2} \|u_j\|_X^2 + \int_\Omega a_n(x) G(u_j) - \int_\Omega f_n(x) u_j  \right) \\ \notag 
    &\ge \frac{1}{2\Lambda} \|u\|_X^2 + \int_\Omega a_n(x) G(u) - \int_\Omega f_n(x) u.
\end{align}
So,
\begin{equation*}
 \int_\Omega a_n(x) G(u) \le \lim_{j\to+\infty} J_n(u_j) + \int_\Omega f_n(x) u \le \|f_n\|_{L^2(\Omega)}\|u\|_{L^2(\Omega)} + \lim_{j\to+\infty} J_n(u_j) \underset{\eqref{eq_limit}}{<} +\infty,   
\end{equation*}
which implies $a_n G(u)\in L^1(\Omega).$ 

In order to conclude, we observe that the quadratic form
\begin{equation} \label{eq:3.11}
    u \mapsto \Vert u \Vert^2 = \int_\Omega B(x) \nabla u \nabla u
\end{equation}
is weakly lower semicontinuous on $\mathscr{H}_X(\Omega)$. This is actually a consequence of the positivity of this form, see for instance \cite{Hestenes}, but we provide the details for the reader's convenience.
By \eqref{eq_X_ell} it follows
$$
\frac{1}{\Lambda} \|u\|_X^2 \le \|u\|^2 \le \Lambda\|u\|_X^2 \quad \hbox{for each $u \in \mathscr{H}_X(\Omega)$,}
$$
hence the quadratic form \eqref{eq:3.11} is a norm on $\mathscr{H}_X(\Omega)$ equivalent to $\|\cdot\|_X$.
So $\|u\|\le \liminf_{j\to+\infty}\|u_j\|$, that is
\begin{equation}
    \label{eq_estim_weak}
    \liminf_{j\to+\infty}\int_\Omega B(x)\nabla u_j \nabla u_j \ge \int_\Omega B(x) \nabla u \nabla u.
\end{equation}
By \eqref{eq_estim_weak} and the previous estimates \eqref{eq_convergence1}-\eqref{eq_convergence2}
\begin{align*}
    \lim_{j\to+\infty} J_n(u_j) &\ge \liminf_{j\to+\infty} \left(\frac{1}{2}\int_\Omega B(x)\nabla u_j \nabla u_j + \int_\Omega a_n(x) G(u_j) - \int_\Omega f_n(x) u_j \right)
    \\ &\ge   \frac{1}{2} \int_\Omega B(x) \nabla u \nabla u +\int_\Omega a_n(x) G(u) - \int_\Omega f_n(x) u 
    \\&= J_n(u), 
\end{align*}
which is in contradiction with \eqref{eq_limit}.

\noindent\textbf{Step 3.} The functional $J_n$ is coercive, namely $J(u) \to +\infty$ as $\|u\|_X \to +\infty$.
By the definition of $J_n$ in \eqref{eq_functional}, \eqref{eq_positive} and \eqref{eq_elliptic} it follows that
\begin{align*}
    J_n(u) &\ge \frac{1}{2\Lambda} \| u\|_X^2 +\int_\Omega a_n(x) G(u) - \int_\Omega f_n(x) u \\
    &\ge  \frac{1}{2\Lambda} \| u\|_X^2 - \int_\Omega f_n(x) u \\
    & \ge  \frac{1}{2\Lambda} \| u\|_X^2 - \int_\Omega |f_n(x)| |u|.
\end{align*}
Now by H\"{o}lder's inequality and the Sobolev embedding \ref{sobolev}
\begin{equation*}
J_n(u)\ge  \frac{1}{2\Lambda} \| u\|_X^2 - S^{-1/2}\|f_n\|_{L^{\frac{2N_X}{N_X+2}}(\Omega)}\|u\|_X,    
\end{equation*}
where $S$ is the Sobolev constant defined in \eqref{eq_sob_const} and $N_X$ is defined in \eqref{eq_dimension}. Passing to the limit as $\|u\|_X\to+\infty$, we can conclude.

\noindent\textbf{Step 4.} There exists a solution $u_n\in\mathscr{H}_X(\Omega)$ to \eqref{problem_approx}.

As an application of the Direct Method in the Calculus of Variations (see e.g. \cite[Theorem 5.5]{AbrosettiMalchiodi}), the weakly lower semicontinuous and coercive functional $J_n$ admits a critical point $u_n\in\mathscr{H}_X(\Omega)$, which is a solution of \eqref{problem_approx}.
\end{proof}

\begin{rmk}
    Although the approximated problem was originally solved in \cites{AB15, AB21} by means of a fixed-point argument, in our case the discussion of \cite[Remark 2.8]{ADLSV25} is in force and would lead us to a rather involved proof based on a truncation of the nonlinearity $g$. As we have shown, a variational approach seems to be much easier for our purposes.
\end{rmk}
We can now prove the existence of a solution of \eqref{problem}.
\begin{proof}[Proof of Theorem \ref{thm_existence_1}]  
Reasoning as in \cite[Corollary 2.2]{GL03} we can check that $G_k(u_n)$ can be used as a test function in the approximated problem \eqref{problem_approx}. This gives, thanks to \eqref{eq_Q_cond_approx} and \eqref{eq_elliptic}
\begin{equation*}
  \frac{1}{\Lambda} \int_\Omega |X G_k(u_n)|^2 + \int_\Omega a_n(x) g(u_n) G_k(u_n)\le \int_\Omega |f_n(x)| |G_k(u_n)| \le Q \int_\Omega a_n(x) |G_k(u_n)|.    
\end{equation*}
Since it is easily seen that $G_k(u_n) u_n = |G_k(u_n)| |u_n|$, we get
\begin{equation}
    \label{eq_ineq_1}
      \frac{1}{\Lambda} \int_\Omega |XG_k(u_n)|^2 + \int_\Omega a_n(x) \left[|g(u_n) - Q| \right] |G_k(u_n)|\le 0.
\end{equation}
By Remark \ref{rmk:1.1} and since $Q\in \left(0,\lim_{s\to+\infty}g(s)\right)$, we can choose $k=g^{-1}(Q)$ in \eqref{eq_ineq_1}. As a consequence $g(u_n)-Q\ge 0$ and $G_k(u_n) =0$, that is the sequence $(u_n)$ is bounded in $L^\infty(\Omega)$ with $\|u_n\|_{L^\infty(\Omega)}\le g^{-1}(Q)$. 

So, we can use $u_n$ as a test function in the weak formulation of \eqref{problem}, getting
\begin{align*}
    \frac{1}{\Lambda} \int_\Omega |Xu_n|^2 &\le   \frac{1}{\Lambda}\int_\Omega |Xu_n|^2 + \int_\Omega a_n(x) g(u_n) \le \int_\Omega B(x)\nabla u_n \nabla u_n + \int_\Omega a_n(x) g(u_n) \\&= \int_\Omega f_n(x) u_n\le Q\int_\Omega a_n(x) u_n \le Q g^{-1}(Q) \int_\Omega a(x).
\end{align*}
So, the sequence $(u_n)$ is bounded in $\mathscr{H}_X(\Omega)$. Thus, there exists $u\in\mathscr{H}_X(\Omega)$ and a subsequence -- still denoted by $(u_n)$ -- which converges weakly in $\mathscr{H}_X(\Omega)$ and a.e. to $u$ with
\begin{equation*}
    \|u\|_{L^\infty(\Omega)}\le g^{-1}(Q).
\end{equation*}
Moreover, using that $|a_n(x)g(u_n)|\le a(x) Q$, we obtain by the dominated convergence theorem the $L^1(\Omega)$ convergence of the sequence $(a_n(x)g(u_n))$ to $a(x)g(u)\in L^1(\Omega)$, which together with the $L^1(\Omega)$ convergence of $f_n(x)$ to $f(x)$ allows to pass to the limit in the approximated problem to conclude, also thanks to Lemma \ref{lemma_convergence}, that $u$ is a weak solution of $\eqref{problem}$.

To show that the solution is unique, suppose that $u_1, u_2\in L^\infty(\Omega)\cap\mathscr{H}_X(\Omega)$ are two solution of \eqref{problem}. Using $u_1-u_2$ as a test function in the weak formulation of \eqref{problem}, one obtains that
\begin{equation*}
  \frac{1}{\Lambda} \|u_1 - u_2\|_X^{2} \le \int_\Omega B(x) \nabla(u_1-u_2)\nabla(u_1-u_2) + \int_\Omega a(x) \left(g(u_1) - g(u_2) \right) (u_1-u_2) = 0.    
\end{equation*}
Observe that $\left(g(u_1) - g(u_2) \right) (u_1-u_2) \geq 0$ by assumption \eqref{eq_cond_g}. Recalling that $a(x)\ge 0$, we deduce $u_1 = u_2$.
\end{proof}

\begin{rmk}
A simple example of a function $g$ satisfying \eqref{eq_cond_g} is given by $g(s) = |s|^{\gamma-1}s$ for all $\gamma > 0$. The case $\gamma = 1$ corresponds to the linear case, generalizing the result in \cite[Theorem 2.1]{AB21}.    
\end{rmk}

We conclude this section observing that it is possible to obtain a bounded solution of \eqref{problem} even when $g$ is not increasing. This can be achieved if we require that
\begin{equation}
    \label{eq_cond_g_2}\tag{H$'_g$}
  \hbox{  $g$ is continuous and odd, $\lim_{s\to+\infty} g(s) = +\infty$ and $\int_0^t g(\tau)\,\mathrm{d}\tau\ge 0$ for all $t\in\R$}.
\end{equation} 
A simple model illustrating this scenario is given by the family of functions $g(s) = a \sin s + b s$, for all $a > b > 0$. 
\begin{prop}
    If $\eqref{eq_cond_a}$-\eqref{eq_cond_g_2} and \eqref{eq_Q_cond} are satisfied, then there exists a weak solution $u\in \mathscr{H}_X(\Omega)\cap L^\infty(\Omega)$ to \eqref{problem}.
\end{prop}
\begin{proof}
The proof is very similar to that of Theorem \ref{thm_existence_1}, so we only highlight the main differences.

The existence of a solution $u_n\in\mathscr{H}_X(\Omega)$ of the approximated problem \eqref{problem_approx} is guaranteed by Proposition \ref{prop_approx}. 
Now, by \eqref{eq_cond_g_2}, we can choose $k_0>0$ such that
    \begin{equation}
        \label{eq_bound_g}
        g(s)\ge Q \quad\text{ for every $s\ge k_0$.}
    \end{equation}
Then, using $G_{k_0}(u_n)$ as a test function in the approximated problem gives, with the same calculations as in the proof of Theorem \ref{thm_existence_1},
\begin{equation*}
  \frac{1}{\Lambda}\int_\Omega |X G_{k_0}(u_n)|^2 + \int_\Omega a_n(x) \left[|g(u_n)| - Q  \right] |G_{k_0}(u_n)| \le 0,    
\end{equation*}
which implies $\|u_n\|_{L^\infty(\Omega)}\le k_0$ and the sequence $(u_n)$ is bounded in $L^\infty(\Omega)$.

So, using $u_n$ as a test function in \eqref{problem_approx},
\begin{equation*}
  \frac{1}{\Lambda} \int_\Omega |Xu_n|^2 - \max_{|s|\le k_0} |g(s)s| \int_\Omega a_n(x)\le k_0 \int_\Omega |f_n(x)|,    
\end{equation*}
that is
\begin{equation*}
  \frac{1}{\Lambda} \|u_n\|_X^2 \le k_0 \int_\Omega |f| + \max_{|s|\le k_0} |g(s)s|\int_\Omega a(x).    
\end{equation*}
This proves that $(u_n)$ is bounded in $\mathscr{H}_X(\Omega)$, so there exists a function $u \in \mathscr{H}_X(\Omega)$ such that, up to a subsequence, $u_n \rightharpoonup u$ in $\mathscr{H}_X(\Omega)$ and $u_n\to u$ a. e. in $\Omega$.

Now we pass to the limit as $n\to+\infty$ as in Theorem \ref{thm_existence_1}. The dominated convergence theorem can still be applied to show that $a_ng(u_n) \to a g(u)$ in $L^1(\Omega)$, since we have
\[
|a_n(x)g(u_n)| \leq a(x) \max_{|s|\leq k_0} |g(s)| \in L^1(\Omega).
\]
Passing to the limit, it follows that $u$ is a bounded weak solution of \eqref{problem}.
\end{proof}

\begin{rmk}
    Since $g$ is not increasing, we cannot guarantee the uniqueness of the solution as before. 
    By the definition of $k_0$ in \eqref{eq_bound_g} we can, however, give a bound of the $L^\infty$-norm of the solutions related to the function $g$ and the constant $Q$:
    \begin{equation*}
          \|u\|_{L^\infty(\Omega)} \le \inf \{k>0 : g(s)\ge Q \hbox{ for every $s\ge k$}\}.
    \end{equation*}
    The latter estimate coincides with \eqref{eq_bound_u} when $g$ is increasing.
\end{rmk}

\section{Generalized $Q$-condition for linear problems}
\label{sec_linear}
In this section, we study the regularizing effect of condition \eqref{eq_Q_cond_2} in linear problems involving $X$-elliptic operators. In particular, we consider the problem
\begin{equation} \tag{$\tilde{P}$}\label{problem_2}
\left\{ 
    \begin{alignedat}{2}
        -\mathcal{L}u + a(x) u & = f(x) \quad &&\mbox{in} \;\; \Omega, \\ 
        u &= 0 \quad &&\mbox{on} \;\; \partial\Omega, 
    \end{alignedat} 
\right. 
\end{equation}
where $\mathcal{L}$ is defined in \eqref{eq_def_operator}. We suppose that $\mathcal{L}$ is an $X$-elliptic operator in the sense of Definition \ref{def_X_elliptic} and let $\Omega$ be a bounded subset of $\R^N$ $(N\ge 2)$.
Let now suppose that there exists a function $0\le R \in L^{\frac{2N_X}{N_X +2}}(\Omega)$, where $N_X$ is defined in \eqref{eq_dimension}, such that \eqref{eq_Q_cond_2} is verified.

\begin{rmk}
We observe that the generalized $Q$-condition \eqref{eq_Q_cond_2} holds also for $f_n$ and $a_n$, defined in \eqref{eq_func_approx}. In fact, using again the monotonicity of the function $s\mapsto s/\left(1+\frac{s}{n}\right)$ we get
\begin{equation}
    \label{eq_Q_cond_gen_approx} 
    |f_n(x)|= \frac{|f(x)|}{1+\frac{1}{n}|f(x)|}\le \frac{Q a(x) +R(x)}{1+\frac{1}{n}\left( Qa(x) + R(x) \right)}\le Q a_n(x) + R(x).
\end{equation}
\end{rmk}
The following existence result holds.
\begin{theorem} \label{thm_existence_linear}
    Suppose \eqref{eq_cond_a} and \eqref{eq_Q_cond_2}. Then problem \eqref{problem_2} admits a unique weak solution $u\in\mathscr{H}_X(\Omega)$.
\end{theorem}

\begin{proof}
The existence of a weak solution $u_n\in\mathscr{H}_X(\Omega)$ to the approximated problem, i.e. a function $u_n\in\mathscr{H}_X(\Omega)$ such that 
    \begin{equation}
    \label{prob_approx_2}
        \int_\Omega B(x)\nabla u_n\nabla v + \int_\Omega a_n(x) u_n v = \int_\Omega f_n(x) v, \quad \forall v\in\mathscr{H}_X(\Omega)
    \end{equation}
    is given by Proposition \ref{prop_approx} with $g(s) = s$, since it is a general result that does not require any specific condition on $a_n$ and $f_n$.

 Consider the sequence $(u_n)$ of solutions of \eqref{prob_approx_2}. Choosing $u_n $ as a test function in \eqref{prob_approx_2} we get by \eqref{eq_elliptic} and \eqref{eq_Q_cond_gen_approx}
    \begin{equation}
        \label{eq_bound_2}
    \frac{1}{\Lambda}\int_\Omega |Xu_n|^2 + \int_\Omega a_n(x) u_n^2 \le Q \int_\Omega a_n(x) |u_n| + \int_\Omega R(x) |u_n|.
    \end{equation}
    Observe that $Q < |u_n|$ implies $Q |u_n| < |u_n|^2$, hence
\begin{align*}
    Q\int_\Omega a_n(x) |u_n| &= Q \int_{|u_n|\le Q} a_n(x) |u_n| + Q \int_{|u_n|>Q} a_n(x) |u_n|\\
   & \le Q \int_{|u_n|\le Q} a_n(x) |u_n| + \int_{|u_n|>Q} a_n(x) |u_n|^2,
\end{align*}
    we get, after bringing the term $\int_{|u_n|>Q} a_n(x) |u_n|^2$ to the left-hand side of \eqref{eq_bound_2},
    \begin{align*}
           \frac{1}{\Lambda} \int_\Omega |X u_n|^2
            &\le \frac{1}{\Lambda} \int_\Omega |Xu_n|^2 + \int_{|u_n|\le Q} a_n(x) |u_n|^2
            \\&
            \le Q\int_{|u_n|\le Q} a_n(x) |u_n| + \int_\Omega R(x) |u_n|
           \\ &
           \le Q^2 \int_\Omega a_n(x) + \int_\Omega R(x) |u_n|.
    \end{align*}
Now by H\"{o}lder's inequality and the Sobolev embedding \ref{sobolev}
\begin{equation*}
\frac{1}{\Lambda}\|u_n\|_X^2 \le Q^2\int_\Omega a(x) + S^{-1/2} \| R\|_{L^{\frac{2N_X}{N_X+2}}(\Omega)} \|u_n\|_X,    
\end{equation*}
where $S$ is defined in \eqref{eq_sob_const}.
So the sequence $(u_n)$ is bounded in $\mathscr{H}_X(\Omega)$ and then there exist $u\in\mathscr{H}_X(\Omega)$ and a subsequence -- still denoted by $(u_n)$ -- such that $u_n\rightharpoonup u$ in $\mathscr{H}_X(\Omega)$ and $u_n \to u$ a.e. in $\Omega$.

Now, given $\delta, k>0$, we define
\begin{equation*}
   \psi_{k,\delta}(s) \coloneqq
\begin{cases}
    0, & \hbox{if $|s|\le k,$}\\
    \frac{1}{\delta}(s-\operatorname{sgn}(s)k), & \hbox{if $k<|s|<k+\delta$,}\\
    1, & \hbox{if $|s|\ge k+\delta$.}
\end{cases}  
\end{equation*}
We have
\begin{equation*}
\nabla \psi_{k,\delta}(u_n) =
\begin{cases}
    0, & \hbox{if $|u_n|\le k ~\vee ~ |u_n|\ge k+\delta$,}\\
    \frac{1}{\delta}\nabla u_n & \hbox{if $k<|u_n|<k+\delta$,}
\end{cases}    
\end{equation*}
and we use $\psi_{k,\delta}(u_n)$ as a test function in \eqref{prob_approx_2}. We drop the positive term $\frac{1}{\Lambda\delta}\int_{k< |u_n|< k+\delta} |Xu_n|^2$ in \eqref{eq_elliptic} to get
\begin{equation*}
\int_\Omega a_n(x) u_n \psi_{k,\delta} (u_n) \le \int_\Omega |f_n(x)| |\psi_{k,\delta}(u_n)|\le \int_\Omega |f(x)| |\psi_{k,\delta}(u_n)|.    
\end{equation*}
Since trivially $\chi_{\{|u_n|\ge k+\delta\}}\le |\psi_{k,\delta}(u_n)|\le \chi_{\{|u_n|\ge k\}}$, we obtain
\begin{equation*}
\int_{|u_n|\ge k+\delta } a_n(x) |u_n| \le \int_\Omega a_n(x) u_n \psi_{k,\delta}(u_n) \le \int_\Omega |f(x)||\psi_{k,\delta}(u_n)| \le \int_{|u_n|\ge k} |f(x)|.    
\end{equation*}
Letting $\delta\to 0$, Fatou's lemma yields
\begin{equation*}
\int_{|u_n|\ge k} a_n(x) |u_n| \le \int_{|u_n|\ge k} |f(x)|,    
\end{equation*}
so that, for every measurable subset $E\subseteq\Omega$, we have
\begin{align*}
\int_E a_n(x) |u_n| &= \int_{E\cap \{ |u_n|< k \} } a_n(x) |u_n| +  \int_{E\cap \{ |u_n| \ge k \} } a_n(x) |u_n|   \\&\le k \int_E a_n(x) + \int_{|u_n|\ge k} a_n(x) |u_n| \\& \le k \int_E a(x) + \int_{|u_n|\ge k} |f(x)|.
\end{align*}
Since $a$, $f\in L^1(\Omega)$ and $(u_n)$ is bounded in $L^1(\Omega)$, by the absolute continuity of the integral, see \cite[Proposition 16.3]{Secchi}, for any $\varepsilon>0$ we can pick $k$ so large that
\begin{equation*}
\int_{|u_n|\ge k} |f(x)|<\varepsilon.    
\end{equation*}
Then
\begin{equation*}
\lim_{|E|\to 0} \int_E a_n(x) |u_n|\le \varepsilon,    
\end{equation*}
uniformly with respect to $n\in\N$, and we can use Vitali theorem (see e.g. \cite[Theorem 4.5.4]{Bogachev}) to prove that
\begin{equation*}
    a_n(x)u_n \to a(x) u \quad  \mbox{in $L^1(\Omega)$}.    
\end{equation*}
Now, passing to the limit in \eqref{prob_approx_2} and using Lemma \ref{lemma_convergence}, we have that $u$ is a weak solution of \eqref{problem_2}, in the sense of Definition \ref{def_weak_sol}.

To prove the uniqueness of the solution, let $u_1, u_2\in\mathscr{H}_X(\Omega)$ two weak solutions of \eqref{problem_2}. Observe that
\begin{align*}
\int_\Omega a(x) (u_1-u_2) T_k(u_1-u_2) &= \int_{|u_1-u_2|<k} a(x)(u_1-u_2)(u_1-u_2) + k \int_{ \{u_1-u_2 \ge k\}} a(x) (u_1-u_2) \\&\quad+k \int_{\{u_2-u_1\ge k \}} a(x) (u_2-u_1)\ge 0,
\end{align*}
since $a(x)\ge 0$. 
We now use $T_k(u_1-u_2)\in\mathscr{H}_X(\Omega)\cap L^\infty(\Omega)$ as a test function in \eqref{problem_2}, which gives
\begin{align*}
   0 &= \int_\Omega B(x)\nabla (u_1-u_2) \nabla T_k(u_1-u_2) + \int_\Omega a(x) (u_1-u_2)T_k(u_1-u_2) \\&= \int_{|u_1-u_2|<k} B(x) \nabla (u_1-u_2)\nabla(u_1-u_2) +\int_\Omega a(x) (u_1-u_2)T_k(u_1-u_2) \\&\ge \int_{|u_1-u_2|<k} B(x) \nabla (u_1-u_2)\nabla(u_1-u_2) \ge \frac{1}{\Lambda} \, \bigl\| (u_1 - u_2)\big|_{\{|u_1 - u_2| < k\}} \bigr\|^2.
\end{align*}
So, $u_1 = u_2$ on the set $\{|u_1-u_2|<k \}\subseteq\Omega$, for all $k>0$, which implies $u_1 = u_2$ a.e. in $\Omega$.
\end{proof}

\begin{ex}
The condition \eqref{eq_Q_cond_2} is satisfied, for instance, if
\[
  |f(x)| \leq Q\, a(x)^{\theta}, \quad \text{for some } \theta \in (0,1).
\]
Indeed, in this case we have
\[
  |f(x)| \leq Q\, a(x)^{\theta} \leq Q \bigl( a(x) + 1 \bigr),
\]
so that \eqref{eq_Q_cond_2} holds with $R(x) \equiv Q$.   
\end{ex}

We can also establish the boundedness of the solution of \eqref{problem_2} under additional assumptions on $R(x)$ and $a(x)$.

\begin{prop}
    In addition to the assumptions of Theorem \ref{thm_existence_linear}, we suppose that $R(x)\in L^r(\Omega)$ for some $r>N_X/2$ and that there exists $\alpha>0$ such that
    \begin{equation}
        \label{eq_cond_a_2}
        a(x)\ge\alpha \quad \text{a.e. in }\Omega.
    \end{equation}
Then the solution $u$ obtained in Theorem~\ref{thm_existence_linear} is bounded, i.e. $u \in L^\infty (\Omega)$.
\end{prop}

\begin{proof}
      As a first step we prove that the solutions $u_n$ to the approximated problem \eqref{prob_approx_2} belong to $L^\infty (\Omega)$. 
      Using $G_k(u_n)$ as a test function in \eqref{prob_approx_2} and recalling \eqref{eq_elliptic}, we see that 
\begin{align}
\label{eq_est_b}
      \frac{1}{\Lambda}\int_\Omega |XG_k(u_n)|^2&\le \frac{1}{\Lambda}\int_\Omega |XG_k(u_n)|^2 + \int_\Omega a_n(x) u_n G_k(u_n)\\&\notag \le  \int_\Omega B(x) \nabla u_n \nabla G_k(u_n)  + \int_\Omega a_n(x) u_n G_k(u_n)\\&\notag\le    \int_\Omega |f_n(x)| |G_k(u_n)|,
\end{align}
where we have dropped the positive term $\int_\Omega a_n(x) u_n G_k(u_n)$ since $a\ge 0 $ and $u_n G_k(u_n) = |u_n||G_k(u_n)|$. By the Sobolev embedding \ref{sobolev}, it follows
\begin{equation*}
      \frac{S}{\Lambda} \|G_k(u_n)\|_{L^{2^*_X}(\Omega)}^2\le \int_\Omega |f_n(x)| |G_k(u_n)|,    
\end{equation*}
      where the constant $S$ was defined in \eqref{eq_sob_const}.
     To simplify notation, for $k>0$ we set
        \[
        A_k \coloneqq \{ x \in \Omega \mid |u_n(x)| > k \}
        \]
      and we fix $M>\frac{2N_X}{N_X+2}$. Applying H\"{o}lder's inequality twice we get
      \begin{align}
      \label{eq_est_b_2}
          \int_\Omega |f_n(x)| |G_k(u_n)| &=  \int_{A_k} |f_n(x)| |G_k(u_n)| \le \| G_k(u_n) \|_{L^{2^*_X}(\Omega)} \left(  \int_{A_k} |f_n(x)|^{\frac{2N_X}{N_X+2}}\right)^\frac{N_X+2}{2N_X}\\&\le \|f_n\|_{L^M(\Omega)} \|G_k(u_n)\|_{L^{2^*_X}(\Omega)} \left\lvert A_k \right\rvert^{\left[1-\frac{2N_X}{(N_X -2) M}\right]\frac{N_X+2}{2N_X}}.\notag
      \end{align}
      Estimates \eqref{eq_est_b}-\eqref{eq_est_b_2} give
      \begin{equation*}
      \frac{S}{\Lambda} \|G_k(u_n)\|_{L^{2^*_X}(\Omega)}^2\le \|f_n\|_{L^M(\Omega)} \|G_k(u_n)\|_{L^{2^*_X}(\Omega)} \left\lvert A_k \right\rvert^{\left[1-\frac{2N_N}{(N_X -2)M}\right]\frac{N_X+2}{2N_X}},          
      \end{equation*}
      that is 
      \begin{equation}
          \label{eq_est_b_3}
                \frac{S}{\Lambda} \|G_k(u_n)\|_{L^{2^*_X}(\Omega)}\le \|f_n\|_{L^M(\Omega)} \left\lvert A_k \right\rvert^{\left[1-\frac{2N_N}{(N_X -2)M}\right]\frac{N_X+2}{2N_X}},
      \end{equation}
      Applying again H\"{o}lder's inequality one has
      \begin{equation*}
      \int_\Omega |G_k(u_n)| \le \left(\int_\Omega |G_k(u_n)|^{2^*_X}\right)^\frac{1}{^{2^*_X}} \left\lvert A_k  \right\rvert^{\frac{N_X+2}{2N_X}}.          
      \end{equation*}
      So \eqref{eq_est_b_3} implies that
      \begin{equation*}
      \int_\Omega |G_k(u_n)|\le \|f_n\|_{L^M(\Omega)} |A_k|^{1+\frac{2}{N_X} - \frac{1}{M}}.        
      \end{equation*}
      Owing to the choice of $M$, we have that $1+\frac{2}{N_X} - \frac{1}{M}>1$ and so we can apply \cite[Lemma 6.2]{BoccardoCroce}, to obtain that 
      \begin{equation*}
      \| u_n\|_{{L^\infty}(\Omega)}\le C\|f_n\|_{L^M(\Omega)},  
      \end{equation*}
    where $C$ depends on $\Omega$ and $N_X$.

    \medskip
    Now we prove that the sequence $(u_n)$ is bounded in $L^\infty(\Omega)$. For a given number $k>Q$ we let 
    \begin{equation*}
      v_n \coloneqq (e^{2|G_k(u_n)|}-1)\operatorname{sgn}(u_n). 
    \end{equation*}
    An easy computation shows that $\nabla v_n = 2 e^{2 |G_k(u_n)|}\nabla G_k(u_n)$ and so, using $v_n$ as a test function in \eqref{prob_approx_2}, we get
    \begin{equation*}
      2\int_\Omega B(x) \nabla u_n e^{2 |G_k(u_n)|} \nabla G_k(u_n) + \int_\Omega a_n(x) |u_n| (e^{2|G_k(u_n)|}-1) \le \int_\Omega |f_n(x)| (e^{2|G_k(u_n)|}-1).  
    \end{equation*}
Therefore, by \eqref{eq_elliptic} and \eqref{eq_Q_cond_gen_approx}, 
    \begin{align}
\notag
        \frac{2}{\Lambda}\int_\Omega | XG_k(u_n)|^2 e^{2 |G_k(u_n)|} &+ \int_\Omega a_n(x) \left(e^{2 |G_k(u_n)|}-1 \right) [|u_n|-Q] \le \int_\Omega R(x) \left(e^{2 |G_k(u_n)|}-1 \right)
            \label{eq_est_b_4}
        \\ &\le t \int_{R(x)\le t} \left(e^{2 |G_k(u_n)|}-1 \right) + \int_{R(x)>t} R(x) \left(e^{2 |G_k(u_n)|}-1 \right)
        \\ & \le t \int_\Omega \left(e^{2 |G_k(u_n)|}-1 \right) + \int_{R(x)>t} R(x) \left(e^{2 |G_k(u_n)|}-1 \right)
        \notag
    \end{align}
    for all $t>0$. 
Now, note that 
\begin{equation*}
  \left|X G_k(u_n)\right|^2 e^{2 |G_k(u_n)|} = \left|X (e^{ |G_k(u_n)|})\right|^2 = \left|X (e^{ |G_k(u_n)|}-1)\right|^2  
\end{equation*}
and  the Sobolev embedding \ref{sobolev} yields
\begin{equation}
    \label{eq_est_b_5}
    \int_\Omega |X G_k(u_n)|^2 e^{2 |G_k(u_n)|} \ge S \left[ \int_\Omega \left( e^{|G_k(u_n)|}-1  \right)^{2^*_X} \right]^{\frac{2}{2^*_X}}.
\end{equation}
Now, by \eqref{eq_cond_a_2} we get that for all $n\in\N$ and all $x \in \Omega$
\begin{equation}
    \label{eq_est_b_6}
    a_n(x) \ge \frac{a(x)}{1+Q a(x)}\ge \frac{\alpha}{1+Q\alpha},
\end{equation}
since the function $s\mapsto s/(Q+s)$ is increasing.

Plugging \eqref{eq_est_b_5} and \eqref{eq_est_b_6} into \eqref{eq_est_b_4} we obtain
\begin{align*}
    S \frac{2}{\Lambda} \left[ \int_\Omega \left( e^{|G_k(u_n)|}-1  \right)^{2^*_X} \right]^{\frac{2}{2^*_X}} &+ \int_\Omega \left( e^{2|G_k(u_n)|}-1 \right) \left[\frac{\alpha}{1+Q\alpha}\left(|u_n-Q|\right)-t\right]\\
    &\quad\le \int_{R(x)>t} R(x) \left( e^{2|G_k(u_n)|}-1 \right).
\end{align*}
The elementary inequality
\begin{displaymath}
    e^{2|x|}-1 \leq 2 \left( e^{|x|}-1 \right)^2 +1
\end{displaymath}
valid for each $x \in \mathbb{R}$ shows that we can estimate
\begin{align*}
    \int_{R(x)>t} R(x) \left( e^{2|G_k(u_n)|}-1 \right) &= \int_{\{R(x)>t\} \cap \{ |u_n|\ge k \}} R(x) \left( e^{2|G_k(u_n)|}-1 \right)\\&\notag\le 2 \int_{\{R(x)>t\} \cap \{ |u_n|\ge k \}}  R(x) \left[ e^{|G_k(u_n)|}-1 \right]^2 + \int_{\{R(x)>t\} \cap \{ |u_n|\ge k \}} R(x) \\&\le 2\int_{R(x)>t} R(x) \left[ e^{|G_k(u_n)|}-1 \right]^2 + \int_{|u_n|\ge k} R(x),\notag
\end{align*}
to get
\begin{align*}
     S &\frac{2}{\Lambda} \left[ \int_\Omega \left( e^{|G_k(u_n)|}-1  \right)^{2^*_X} \right]^{\frac{2}{2^*_X}} + \int_\Omega \left( e^{2|G_k(u_n)|}-1 \right) \left[\frac{\alpha}{1+Q\alpha}\left(|u_n-Q|\right)-t\right] \\&\le 2 \int_{R(x)>t} R(x) \left( e^{|G_k(u_n)|}-1 \right)^2 + \int_{|u_n|\ge k} R(x)
     \\& \le 2 \left[ \int_{R(x)>t} R(x)^{\frac{N_X}{2}}\right]^{\frac{2}{N_X}}\left[ \int_\Omega \left( e^{|G_k(u_n)|}-1\right)^{2^*_X} \right]^{\frac{2}{2^*_X}} + \int_{|u_n|\ge k} R(x),
\end{align*}
by H\"{o}lder's inequality, i.e.
\begin{multline}
    \label{eq_est_b_7}
    \left(  \frac{2S}{\Lambda}- 2 \left[ \int_{R>t} R^{\frac{N_X}{2}}\right]^{\frac{2}{N_X}}  \right)  \left[ \int_\Omega \left( e^{|G_k(u_n)|}-1  \right)^{2^*_X} \right]^{\frac{2}{2^*_X}} \\
    \quad {} + \int_\Omega \left( e^{2|G_k(u_n)|}-1 \right) \left[\frac{\alpha}{1+Q\alpha}\left(|u_n-Q|\right)-t\right] \le \int_{|u_n|\ge k} R.
\end{multline}
If we now choose $t$ so large that
\begin{equation*}
\frac{2S}{\Lambda}- 2 \left[ \int_{R>t} R^{\frac{N_X}{2}}\right]^{\frac{2}{N_X}}  \ge \frac{S}{\Lambda}>0    
\end{equation*}
and 
$k> Q + \frac{t(1+Q\alpha)}{\alpha}$, then
\begin{equation*}
\int_\Omega \left( e^{2|G_k(u_n)|}-1 \right) \left[\frac{\alpha}{1+Q\alpha}\left(|u_n-Q|\right)-t\right] \ge 0.    
\end{equation*}
This implies by \eqref{eq_est_b_7} that
\begin{equation*}
C \left[\int_{|u_n|\ge k} \left( e^{(|u_n|-k)}-1 \right)^{2^*_X}\right]^{\frac{2}{2^*_X}} \le \int_{|u_n|\ge k} R    
\end{equation*}
for a suitable constant $C>0$ not depending on $u_n$, which implies, by H\"{o}lder's inequality
\begin{equation*}
C \left[\int_{|u_n|\ge k} \left( e^{(|u_n|-k)}-1 \right)^{2^*_X}\right]^{\frac{2}{2^*_X}} \le \|R\|_{L^r(\Omega)} |\{|u_n|\ge k \}|^{\frac{1}{r'}},    
\end{equation*}
with $\frac{1}{r} + \frac{1}{r'}=1$.

Let $h>k$ so that we deduce
\begin{equation*}
(h-k)^2 \left\lvert \{ |u_n|\ge h \}  \right\rvert^{\frac{2}{2^*_X}} = \left[ \int_{|u_n|\ge h} (h-k)^{2^*_X}  \right]^{\frac{2}{2^*_X}} \le \tilde{C}  |\{|u_n|\ge k \}|^{\frac{1}{r'}},    
\end{equation*}
that is, setting $\varphi(s)\coloneqq |\{ |u_n|\ge s \}  |$,
\begin{equation*}
\varphi(h) \le \frac{\hat C}{(h-k)^{2^*_X}} \varphi(h)^{\frac{1}{r'}\cdot \frac{2^*_X}{2}}.    
\end{equation*}
The choice of $r>N_X/2$ yields $\frac{1}{r'}\cdot \frac{2^*_X}{2}>1$, and we can apply Stampacchia's lemma, \cite[Lemma 4.1, p.19]{Stampacchia} to deduce that there exists $d>0$ such that $\varphi(d) = 0$, that is
\begin{equation*}
\|u_n\|_{L^\infty(\Omega)} \le d,    
\end{equation*}
for all $n\in \N$.

We can now conclude the proof since, as proved in Theorem \ref{thm_existence_linear}, the sequence $(u_n)$ converges in $\mathscr{H}_X(\Omega)$ and a.e. to a function $u\in\mathscr{H}_X(\Omega)$, which is a solution of \eqref{problem_2}.
\end{proof}

\section*{Acknowledgments}
\noindent
This work has been funded by the European Union - NextGenerationEU
within the framework of PNRR Mission 4 - Component 2 - Investment 1.1
under the Italian Ministry of University and Research (MUR) program
PRIN 2022 - grant number 2022BCFHN2 - Advanced theoretical aspects in
PDEs and their applications - CUP: H53D23001960006 and partially
supported by the INdAM-GNAMPA Research Project 2024: Aspetti
geometrici e analitici di alcuni problemi locali e non-locali in
mancanza di compattezza - CUP: E53C23001670001.

\bibliographystyle{plain}
\bibliography{refs}
   
\end{document}